\RequirePackage{fix-cm}
\documentclass{svjour3}                     
\smartqed  
\usepackage{fancyhdr}
\usepackage{graphicx}
\usepackage{amsfonts}
\usepackage{color}
\usepackage{commath}
\usepackage{amsmath}              
\usepackage{latexsym}
\usepackage{subfigure}
\usepackage{url}
\usepackage{natbib}
\setcitestyle{authoryear,open={(},close={)}} 
\begin{document}

\title{Fractional Order Malaria Model With Temporary Immunity 
}


\author{E. Okyere, F. T. Oduro, S. K. Amponsah, I. K. Dontwi, N. K. Frempong
}


\institute{Eric Okyere \at
              Department of Basic Sciences, University of Health and Allied Sciences, PMB 31, Ho, Volta Region, Ghana
\\
              Tel.: +233-508389028\\
              \email{eokyere@uhas.edu.gh}           
          \and
           F. T. Oduro, S. K. Amponsah, I. K. Dontwi, N. K. Frempong\at
           Department of Mathematics, Kwame Nkrumah University of Science and Technology, Kumasi, Ghana
}

\date{March, 2016}

\maketitle

\begin{abstract}
In this paper, we propose and analyze a new fractional order malaria model with temporary immunity. We extend the existing malaria model to include fractional derivatives. Basic reproduction number is computed using the next generation matrix method. Model equilibria are determined and their local asymptotic stability analysis are considered. An efficient Adams-type predictor-corrector method is used to solve the initial value problem. Finally, phase plane portraits are generated to describe the dynamics between susceptible fractions of human hosts and infected fractions of mosquito vectors.
\end{abstract}

\section{Introduction}
Malaria is one of the deadly vector-borne diseases affecting the developing part of the world and mostly in Africa. The reports by the World Health Organization \citep[see, e.g,][]{WHO2005, WHO2012, WHO2014} indicates that, malaria is a serious threat to human life and remains a dangerous infectious disease. Its economic burden on the affected countries are extremely huge and needs serious public health attention. Since its discovery, many researchers have studied the disease dynamics mathematically to capture and understand the complex interaction between infected mosquitoes and the susceptible human population.\\

Mathematical modeling of infectious diseases has played a key role in understanding the dynamical processes of disease transmission and control strategies.
The first malaria epidemiological model was constructed and formulated by \cite{Ross1911}. He used deterministic compartmental models to describe malaria infection dynamics between human and mosquito populations. \cite{shu2000} proposed a malaria model and considered global stability analysis for the disease free equilibrium point. \cite{Mandal2011} considered an extensive survey and review on modeling of malaria infection. Backward bifurcation analysis is very important in mathematical modeling of infectious diseases and has been investigated by many authors in malaria infection models  \citep[see, e.g,][]{Keegan2013, Gosh2013, Chiyaka2008, Chitnis2005, Chitnis2006, Ngonghala2012}. \cite{Aguas2012} developed an age structure model to study the effects of relapse in malaria parasites. A periodic model based on the Ross-MacDonald model was considered by \cite{Gao2014}. They studied the effects of spatial heterogeneity and temporal on disease dynamics. Their proposed model incorporated periodic behavior in mosquito ecology and seasonal human movement to capture variations of malaria spread among different regions. \\

\cite{Tumwiine2010} proposed and analyzed a malaria infection model with infective immigrants. \cite{Tumwiine2007b} formulated an epidemic model with standard incidence for the dynamics of malaria infection. They applied additive compound matrices approach to show global stability analysis of the endemic equilibrium. Using the model in \cite{Tumwiine2007b}, the same authors \citep{Tumwiine2007a} studied the oscillatory behavior of malaria disease in a population with temporary immunity. Their numerical results shows that the endemic equilibrium converges to a steady state.  \cite{Okosun2011} developed a deterministic compartmental model to investigate the impact of drug resistance in malaria infection and the same authors \citep{makinde2014} formulated an optimal control model for malaria and cholera co-infection.  \\


\cite{Abdullahi2013} have assessed the effectiveness of drugs in a malaria transmission model.
\cite{Barley2012}, \cite{Mukandavire2009} and \cite{Abu-Raddad2006} formulated a model to study infectious disease co-infection of HIV and malaria. \cite{Okosun2013} studied the spread of malaria infection with an optimal control model. \cite{Dondrop2008}, they studied the spread of anti-malaria resistance. \cite{lawi2011} studied co-infection model for malaria and meningitis among children. They observed that, when the threshold parameter $R_{0}<1$, the disease-free equilibrium might exhibit some instabilities in the global sense. They further applied centre manifold theorem to investigate asymptotic stability of the endemic equilibrium. \cite{Anita2008} studied the dynamics of acute malaria transmission using an age-structured models.  A mathematical model for malaria infection which incorporates weather has been formulated and studied by \cite{Moshe2004}.  \cite{Koella2003} developed an epidemic model to explore the incidence of anti-malarial resistance.\\

Mathematical modeling of infectious diseases dynamics and other important areas of studies such as economics, finance and engineering has extensively been explored using the theory and applications of the classical differential equations. But in recent times, the theory and applications of fractional calculus has become extremely useful and important in modeling of biological processes and other areas of studies  due to the memory property of fractional derivatives. Many authors have contributed significantly in compartmental modeling of infectious disease dynamics using fractional differential equations. \cite{wow7} have developed a mosquito-transmitted disease model with fractional differential equations. \cite{wow17} described fractional order models for malaria infection. They modified the integer order compartmental model formulated by \cite{Chiyaka2008} to numerically study control strategies in malaria dynamics. Fractional order models for HIV and TB co-infection have been analysed by \cite{Pinto2014}. Fractional order models for HIV infection have been considered and analysed by several authors \citep[see, e.g,][]{wow24, wow25, wow22, wow15, wow12, wow13, Liu2014}. \cite{wow21} constructed fractional order competition model for love triangle. \cite{wow20} proposed a mathematical model for love using fractional order systems. \cite{wow16} investigated backward bifurcation in a fractional order vaccination model. \\

\cite{Javidi2014} developed fractional order model for cholera infection. \cite{Arafa2012} applied homotopy analysis method and fourth order Runge-Kutta method to generate analytical and numerical results for fractional order childhood disease models. An epidemic model with fractional order for influenza A have been formulated and analysed \citep[see, e.g,][]{El-Shahed2011, Gonzalez2014}. \cite{wow23} studied deterministic fractional order SIRC model with Salmonella bacteria infection. \cite{wow14} presented and discussed fractional order systems for Hepatitis C virus. The deadly ebola disease which killed a lot of people in some part of Africa has been modelled by a system of fractional order derivatives \citep{wow19}. An endemic model with constant population has been analytically and numerically studied by \cite{Okyere2016a} using Caputo fractional derivatives.\\

In this work, we formulate a fractional order model for malaria infection with temporary immunity. We consider the malaria disease model studied and analysed by \cite{Tumwiine2007a, Tumwiine2007b}. The authors developed their model using deterministic integer order differential equations. Our new model will be constructed using fractional order derivatives. We are motivated by this method of mathematical formulation in epidemiological modeling due to the effective nature of fractional derivatives. \\

The rest of the article is as follows. In section 2, we formulate the fractional order malaria model. Section 3, deals with non-negativity of model solutions. In section 4, we compute model equilibria and the basic reproduction number and then investigate local asymptotic stability analysis. We consider numerical simulations in section 5. In section 6, we discuss simulation results and finally conclude the paper in section 7.

\section{Model Derivation}
 In this section, we extend and modify the integer order malaria model by \cite{Tumwiine2007a, Tumwiine2007b} to become fractional order malaria model. The theory and applications of fractional calculus has become extremely useful and important in modeling of biological processes and other areas of studies  due to the memory property of fractional derivatives. Many authors have contributed significantly in compartmental modeling of infectious disease dynamics using fractional differential equations. There are several definitions of fractional derivatives  and integrals in fractional calculus \citep[see, e.g,][]{wow2, wow}, but in this work, our model construction will be based on Caputo derivative. Our motivation for this type of fractional derivative  is based on the fact that, it has an advantage on initial value problems.

\begin{definition} \citep{wow, wow2}
Fractional integral of order $\alpha$ is defined as
\[I^{\alpha}g(t)=\frac{1}{\Gamma(\alpha)}\int_0^t \! \frac{g(x)}{(t-x)^{1-\alpha}} \, \mathrm{d}x\]
\end{definition}
for $0<\alpha<1,\ t>0.$

\begin{definition} \citep{wow, wow2}
Caputo fractional derivative is defined as
\[ D^{\alpha}g(t)=\frac{1}{\Gamma(k-\alpha)}\int_0^t \! \frac{g^{k}(x)}{(t-x)^{\alpha+1-k}} \, \mathrm{d}x.\]
\end{definition}
for $k-1<\alpha<k.$\\

In this article, we consider the scaled malaria model model presented by\\ \cite{Tumwiine2007a, Tumwiine2007b}. The authors, transformed their compartmental malaria model ($S_H I_H R_H- S_V I_V$) into systems of proportions ($s_h i_h r_h- s_v i_v$) for the human and mosquito populations describe by systems of ordinary differential equations.

\begin{equation}\label{SIRmalaria1}
\begin{array}{lllll}
\displaystyle \frac{d s_{h}}{dt} &=& \displaystyle {\lambda_{h}(1-s_{h})-abm s_h i_v+\nu i_{h}+\gamma r_h+\delta s_h i_h} ,\qquad \\[15pt]
\displaystyle \frac{d i_{h}}{dt} &=& \displaystyle {abm s_h i_v-(\nu +r+\lambda_{h}+\delta )i_h+\delta i^2_{h}}, \qquad  \\[15pt]
\displaystyle \frac{d{r_{h}}}{dt} &=& \displaystyle {r i_{h}-(\gamma+\lambda_{h})r_h+\delta i_h r_h}, \qquad  \\[15pt]
\displaystyle \frac{d s_{h}}{dt} &=& \displaystyle {\lambda_{v}(1-s_v)-ac i_h s_v} ,\qquad \\[15pt]
\displaystyle \frac{d{i_{v}}}{dt} &=& \displaystyle {ac s_v i_h-\lambda_{v}i_v}, \qquad
\end{array}
\end{equation}

 with $s_h(t)+i_h(t)+r_h (t)=1$ and $s_v(t)+i_v(t)=1$,\\

where
\begin{description}
\item  $s_h$:       proportion of susceptible human hosts at time $t$
\item $i_h$:        proportion of infected human hosts at time $t$
\item $r_h$:        proportion of immune human hosts at time $t$
\item $s_v$:        proportion susceptible mosquito vectors at time $t$
\item $i_v$:        proportion of infected mosquito vectors at time $t$\\ \\
\item $a$:          average daily biting rate on man by a single mosquito (infection rate)
\item $b$:          proportion of bites on man that produce an infection
\item $\nu$        recovery rate of human hosts from the disease
\item $c$:          probability that a mosquito becomes infectious
\item $\gamma$:          per capita rate of loss of immunity in human hosts
\item $r$:          rate at which human hosts acquire immunity
\item $\delta$:     per capita death rate of infected human hosts due to the disease
\item $m=\frac{N_V}{N_H}$:      number of female mosquitoes per human host
\item $N_H$:  total human population
\item $N_V$:  total mosquito population
\item $\mu_h$:  per capita natural death rate of humans
\item $\mu_v$:  per capita natural death rate of mosquitoes
\item $\lambda_{h}$:  per capita natural birth rate of humans
\item $\lambda_{v}$: per capita natural birth rate of mosquitoes
\end{description}

The new model we consider in this paper is the system of non-linear fractional differential equations in the sense of Caputo fractional derivatives.

\begin{equation}\label{SIRmalaria2}
\begin{array}{lllll}
\displaystyle D^{\alpha}s_h &=& \displaystyle {\lambda^{\alpha}_{h}(1-s_{h})-a^{\alpha}bm s_h i_v+\nu^{\alpha} i_{h}+\gamma^{\alpha} r_h+\delta^{\alpha} s_h i_h} ,\qquad \\[15pt]
\displaystyle D^{\alpha}i_h &=& \displaystyle {a^{\alpha}bm s_h i_v-(\nu^{\alpha} +r^{\alpha}+\lambda^{\alpha}_{h}+\delta^{\alpha} )i_h+\delta^{\alpha} i^2_{h}}, \qquad  \\[15pt]
\displaystyle D^{\alpha}r_h &=& \displaystyle {r^{\alpha}i_{h}-(\gamma^{\alpha}+\lambda^{\alpha}_{h})r_h+\delta^{\alpha} i_h r_h}, \qquad  \\[15pt]
\displaystyle D^{\alpha}s_v&=& \displaystyle {\lambda^{\alpha}_{v}(1-s_v)-a^{\alpha}c i_h s_v} ,\qquad \\[15pt]
\displaystyle D^{\alpha}i_v &=& \displaystyle {a^{\alpha}c s_v i_h-\lambda^{\alpha}_{v}i_v}, \qquad
\end{array}
\end{equation}

with $s_h(t)+i_h(t)+r_h (t)=1$ and $s_v(t)+i_v(t)=1$.\\

It is important to remark that when the fractional order $\alpha\rightarrow 1$,  the fractional order malaria model~(\ref{SIRmalaria2}) with the model restrictions $s_h(t)+i_h(t)+r_h (t)=1$ and $s_v(t)+i_v(t)=1$, becomes the integer order malaria model~(\ref{SIRmalaria1}) studied by \cite{Tumwiine2007a,  Tumwiine2007b}.

\section{Non-negative Solutions}
Let $\mathbf{R^{5}_{+}}=\{X\in\mathbf{R^{5}}:X\geq 0\}$, where $X=(s_h, i_h, r_h, s_v, i_v )^T$. We apply the following Lemma in \citep{wow8} to show the theorem about the non-negative solutions of the model.

\begin{lemma}\citep{wow8}\label{lema}
(Generalized Mean Value Theorem).\\  Suppose that $w(x)\in C[a, b]$ and $D^{\alpha}w(x)\in C(a, b]$ for $0<\alpha\leq 1$, then we have
\[ w(x)=w(a)+\frac{1}{\Gamma (\alpha)}D^{\alpha}w(\xi)(x-a)^{\alpha}\]
with $a\leq \xi\leq x, \ \forall x\in (a, b].$
\end{lemma}

\begin{remark}
Assume that $w(x)\in C[a,b]$ and  $D^{\alpha}w(x)\in C(a,b]$, for $0<\alpha\leq 1$. It follows from Lemma \ref{lema} that if $D^{\alpha}w(x)\geq 0, \forall x\in (a,b)$ then  $w(x)$ is non-decreasing $\forall x\in [a,b]$ and if $D^{\alpha}w(x)\leq 0, \forall x\in (a, b)$, then $w(x)$ is non-increasing  $\forall x\in [a,b].$
\end{remark}

\begin{theorem}
The malaria disease model~(\ref{SIRmalaria2}) has a unique solution and  it remains in $\mathbf{R^{5}_{+}}$.
\end{theorem}

\begin{proof}
By applying Theorem $3.1$ and Remark $3.2$ in \citep{wow10}, the existence and uniqueness of the solution of equation~(\ref{SIRmalaria2}) in $(0, \infty)$ follows. The domain $\mathbf{R^{5}_{+}}$ for the model problem is positively invariant, since

\begin{equation}\label{SIRmalaria2a}
\begin{array}{lllll}
\displaystyle D^{\alpha}s_h|_{s_h=0}&=& \displaystyle {\lambda^{\alpha}_{h}+\nu^{\alpha} i_{h}+\gamma^{\alpha} r_h\geq 0} ,\qquad \\[15pt]
\displaystyle D^{\alpha}i_h|_{i_h=0} &=& \displaystyle {a^{\alpha}bm s_h i_v\geq 0}, \qquad  \\[15pt]
\displaystyle D^{\alpha}r_h|_{r_h=0} &=& \displaystyle {r^{\alpha}i_{h}\geq 0}, \qquad  \\[15pt]
\displaystyle D^{\alpha}s_v|_{s_v=0}&=& \displaystyle {\lambda^{\alpha}_{v}\geq 0} ,\qquad \\[15pt]
\displaystyle D^{\alpha}i_v|_{i_v=0} &=& \displaystyle {a^{\alpha}c s_v i_h\geq 0}, \qquad
\end{array}
\end{equation}
on each hyperplane bounding the non-negative orthant, the vector filed points into $\mathbf{R^{5}_{+}}$.
\end{proof}

\section{Analysis of Model Equilibria}
In order to determine the model equilibria, we proceed this way.\\

Let

\begin{equation}\label{eqm}
\begin{cases}
D^{\alpha}s_h=0 &\\
D^{\alpha}i_h=0 & \\
D^{\alpha}r_h=0 &\\
D^{\alpha}s_v=0 & \\
D^{\alpha}i_v=0
\end{cases}
\end{equation}

\noindent
By solving equation~(\ref{eqm}) and ignoring the steps involved, the disease-free equilibrium is given by $F_o=(1, 0, 0, 1, 0)$ and the endemic equilibrium as \\ $F_{*}=(s^{*}_{h}, i^{*}_{h}, r^{*}_{h}, s^{*}_{v}, i^{*}_{v})$, where
$s^{*}_{h},  r^{*}_{h},  s^{*}_{v}$  and $i^{*}_{v}$ are expressed in terms $i^{*}_{h}$ as follows:

\begin{equation}\label{SIRmalaria3}
\begin{array}{lllll}
\displaystyle s^{*}_{h} &=& \displaystyle {\dfrac{(\lambda^{\alpha}_{v}+a^{\alpha}c i^{*}_{h})[(\lambda^{\alpha}_{h}+\gamma^{\alpha}-\delta^{\alpha}i^{*}_{h})(\lambda^{\alpha}_{h}+\nu i^{*}_{h})+\gamma^{\alpha}r^{\alpha} i^{*}_{h}]}{(\lambda^{\alpha}_{h}+\gamma^{\alpha}-\delta^{\alpha}i^{*}_{h})[(\lambda^{\alpha}_{h}-\delta^{\alpha }i^{*}_{h})(\lambda^{\alpha}_{v}+a^{\alpha}c i^{*}_{h})+a^{2\alpha}bmci^{*}_{h}]}} ,\qquad \\[15pt]
\displaystyle r^{*}_{h} &=& \displaystyle {\dfrac{r^{\alpha}i^{*}_{h}}{(\lambda^{\alpha}_{h}+\gamma^{\alpha}-\delta^{\alpha}i^{*}_{h})}}, \qquad  \\[15pt]
\displaystyle s^{*}_{v} &=& \displaystyle {\dfrac{\lambda^{\alpha}_{v}}{(\lambda^{\alpha}_{v}+a^{\alpha}c i^{*}_{h})}} ,\qquad \\[15pt]
\displaystyle i^{*}_{v}&=& \displaystyle {\dfrac{a^{\alpha}c i^{*}_{h}}{(\lambda^{\alpha}_{v}+a^{\alpha}c i^{*}_{h})}}, \qquad
\end{array}
\end{equation}

Next, we compute the basic reproduction number $R_o$ using the next generation method on the fractional order malaria model~(\ref{SIRmalaria2}).  By using the method and the notations by \cite{van2002}, the special matrices $F$( for new infections terms) and $V$ (for the remaining transitions terms) related to the model problem~(\ref{SIRmalaria2}) are respectfully given by
\begin{equation}
      F=\begin{bmatrix}
    0 &  & & a^{\alpha}bm \\ \\
    a^{\alpha}c & & &0

\end{bmatrix}
\end{equation}

and

\begin{equation}
      V=\begin{bmatrix}
    \nu^{\alpha}+r^{\alpha}+\lambda^{\alpha}_{h}+\delta^{\alpha}&  & & 0 \\ \\
    0 & & &\lambda^{\alpha}_{v}

\end{bmatrix}
\end{equation}

The basic reproduction number is the spectral radius of the generation matrix $F{V^{-1}}$, where\\

\begin{equation}
      F{V^{-1}}=\begin{bmatrix}
    0&  & &\frac{a^{\alpha}bm}{\lambda^{\alpha}_{v}} \\ \\
    \dfrac{a^{\alpha}c}{\nu^{\alpha}+r^{\alpha}+\lambda^{\alpha}_{h}+\delta^{\alpha}}& & &0

\end{bmatrix}
\end{equation}

Therefore we have\\

$R_{o}=\sqrt{\dfrac{a^{2\alpha}bmc}{{\lambda^{\alpha}_{v}}(\nu^{\alpha}+r^{\alpha}+\lambda^{\alpha}_{h}+\delta^{\alpha})}}$

\subsection{Local Asymptotic Stability Analysis of Disease-Free Equilibrium}
\begin{theorem}
The equilibrium $F_{o}$ of the malaria disease model~(\ref{SIRmalaria2}) is locally asymptotically stable if $R_o<1$ and unstable if $R_o>1.$
\end{theorem}

\begin{proof}

The Jacobian matrix $J(F_{o})$ for the malaria disease model~(\ref{SIRmalaria2}) evaluated at $F_{o}$ is given by

\begin{equation}\label{dfe}
J(F_{o})=\begin{bmatrix}
    -\lambda^{\alpha}_h & \nu^{\alpha}+\delta^{\alpha} & \gamma^{\alpha} & 0  & -a^{\alpha}bm \\
    0 & -(\nu^{\alpha}+r^{\alpha}+\lambda^{\alpha}_h+\delta^{\alpha}) & 0 & 0 & a^{\alpha}bm \\
    0& r^{\alpha}&-(\lambda^{\alpha}_h+\gamma^{\alpha}) &0 & 0 \\
    0 &-a^{\alpha}c & 0 & -\lambda{\alpha}_{v}  & 0\\
    0 &a^{\alpha}c & 0 & 0  & -\lambda^{\alpha}_{v}\\
\end{bmatrix}
\end{equation}

The disease-free equilibrium point $F_o$ is locally asymptotically if all the eigenvalues $\lambda_{i},  i=1, 2, 3, 4, 5$ of $J(F_o)$ satisfy the following condition [\cite{wow5}, \cite{wow6}]: $ \abs{arg(\lambda_{i })} >\frac{\alpha \pi}{2}$.\\

From the square matrix $J(F_{o})$, it is obvious that, three of the eigenvalues  are given by $-\lambda^{\alpha}_h, -(\lambda^{\alpha}_h+\gamma^{\alpha})$ and $-\lambda^{\alpha}_{v}$. It can be seen that all the three eigenvalues are negative and hence satisfy the stability condition: $ \abs{arg(\lambda_{i })} >\frac{\alpha \pi}{2}$.\\

The rest of the eigenvalues can determined from the sub-matrix given as:

\begin{equation}\label{dmt}
      B=\begin{bmatrix}
    -(\nu^{\alpha}+r^{\alpha}+\lambda^{\alpha}_h+\delta^{\alpha}) &  & & a^{\alpha}bm \\ \\
    a^{\alpha}c & & &-\lambda^{\alpha}_{v}

\end{bmatrix}
\end{equation}

The trace  of matrix $B$ is given by
\[tr(B)=-(\nu^{\alpha}+r^{\alpha}+\lambda^{\alpha}_h+\delta^{\alpha}+\lambda^{\alpha}_{v})<0\]

The determinant of $B$ yields :

\begin{align*}
det B &=  \lambda^{\alpha}_{v}(\nu^{\alpha}+r^{\alpha}+\lambda^{\alpha}_h+\delta^{\alpha})-a^{2\alpha}bmc \\ \\
&= \lambda^{\alpha}_{v}(\nu^{\alpha}+r^{\alpha}+\lambda^{\alpha}_h+\delta^{\alpha})\left[1-\dfrac{a^{2\alpha}bmc}{\lambda^{\alpha}_{v}(\nu^{\alpha}+r^{\alpha}+\lambda^{\alpha}_h+\delta^{\alpha})} \ \right]\\ \\
&=\lambda^{\alpha}_{v}(\nu^{\alpha}+r^{\alpha}+\lambda^{\alpha}_h+\delta^{\alpha})\left[1-R^2_{o} \right]
\end{align*}
\end{proof}
Therefore $det B >0 $ if $R_0<1$.\\

Since $det B >0 $ if $R_0<1$ and $tr(B)<0$, it follows that the two remaining eigenvalues will have negative real parts and we can therefore conclude that, the disease equilibrium is locally asymptotically stable, otherwise unstable.

\subsection{Local Asymptotic Stability Analysis of Endemic Equilibrium}
In this subsection, we discuss the local stability of the endemic equilibrium point. Our analysis is based on the construction by \cite{Ahmed2006, wow5}. From the model constraints $s_v+i_v=1$ and $s_h+i_h+r_h=1$, we can reduce the dimension of the malaria model in order to simplify and reduce the complexity of the analysis.\\

As in \cite{Tumwiine2007a, Tumwiine2007b}, let $s_v=1-i_v$ and $r_h=1-s_h-i_h$ then the 5-dimensional fractional order model~(\ref{SIRmalaria2}) reduces to a 3-dimensional problem with Jacobian matrix evaluated the at endemic equilibrium $F_{*}$ as:

\begin{equation*}
J(F_{*})=\begin{bmatrix}
    -(\lambda^{\alpha}_h +\gamma^{\alpha}+a^{\alpha}bmi^{*}_{v}-\delta^{\alpha}i^{*}_{h})& \nu^{\alpha}+\delta^{\alpha}s^{*}_{h}-\gamma^{\alpha}  &- a^{\alpha}bms^{*}_{h}\\ \\
    a^{\alpha}bmi^{*}_{v} & -(\nu^{\alpha}+r^{\alpha}+\lambda^{\alpha}_h+\delta^{\alpha}-2\delta^{\alpha} i^{*}_{h}) & a^{\alpha}bm s^{*}_{h} \\ \\
    0& a^{\alpha}c(1-i^{*}_v)&-(\lambda^{\alpha}_v+a^{\alpha}ci^{*}_h)

\end{bmatrix}
\end{equation*}

The characteristics equation of the Jacobian matrix, $J(F_{*})$ is
\begin{equation}\label{cequation}
  \lambda^{3}+b_1\lambda^{2}+b_2\lambda+b3=0
\end{equation}

where

\begin{equation*}
  \begin{split}
    b_1= &  \lambda^{\alpha}_{v}+a^{\alpha}ci^{*}_{h}+ \lambda^{\alpha}_{h}+\gamma^{\alpha}_{v}+a^{\alpha}bmi^{*}_{v}-\delta^{\alpha}i^{*}_{h}+\nu^{\alpha}+r^{\alpha}+ \lambda^{\alpha}_{h}+\delta^{\alpha}-2\delta^{\alpha}i^{*}_{h}\\
    \\
    b_2= & (\lambda^{\alpha}_v+a^{\alpha}ci^{*}_h)(\nu^{\alpha}+r^{\alpha}+\lambda^{\alpha}_h+\delta^{\alpha}-2\delta^{\alpha} i^{*}_{h})-a^{\alpha}bmc(1-i^{*}_v)s^{*}_h+(\gamma^{\alpha}-\nu^{\alpha}-\delta^{\alpha}s^{*}_h)a^{\alpha}bmi^{*}_v\\
    & +(\lambda^{\alpha}_{h}+\gamma^{\alpha} +a^{\alpha}bmi^{*}_v -\delta^{\alpha} i^{*}_{h})(\lambda^{\alpha}_v+a^{\alpha}ci^{*}_h+\nu^{\alpha}+r^{\alpha}+\lambda^{\alpha}_h+\delta^{\alpha}-2\delta^{\alpha} i^{*}_{h})\\ \\
     b_3= & (\lambda^{\alpha}_{h}+\gamma^{\alpha}+a^{\alpha}bmi^{*}_{v}-\delta^{\alpha}i^{*}_{h})[(\lambda^{\alpha}_{v}+a^{\alpha}ci^{*}_{h})(\nu^{\alpha}+r^{\alpha}+\lambda^{\alpha}_h+\delta^{\alpha}-2\delta^{\alpha} i^{*}_{h})-a^{2\alpha}bmc(1-i^{*}_v)s^{*}_{h}]\\
    & +a^{3\alpha}b^{2}mc(1-i^{*}_v)s^{*}_{h}+a^{\alpha}bmi^{*}_v(\lambda^{\alpha}_{v}+a^{\alpha}ci^{*}_h)(\gamma^{\alpha}-\nu^{\alpha}-\delta^{\alpha}s^{*}_{h})
  \end{split}
\end{equation*}

Let $D(g)$ represent the discriminant of the polynomial function

\[g(x)=x^{3}+b_1x^{2}+b_2x+b3\]\\

where
\begin{equation*}
D(g)=-\left|
\begin{array}{cccccccc}
1 & &b_1 && b_2 & b_3& &0\\
0 & &1& &b_1 & b_2 &&b_3\\
3 & &2 b_1& &b_2 & 0&&0\\
0 & &3& &2b_1&b_2 &&0\\
0 & &0& &3 &2b_1&&b_2
\end{array}
\right|=18b_{1}b_{2}b_{3}+(b_1 b_2)^2-4b_3 b^{3}_{1}-4b^{3}_{2}-27b^{2}_{3}
\end{equation*}

By following the results in \citep{Ahmed2006, wow5} and similar applications of their method in \citep{El-Shahed2011, OZLAP2011, wow40}, we obtain the following proposition.

\begin{proposition} One assume that $F_{*}$ exist in $\mathbf{R^{3}_{+}}$.\\

(i) If the discriminant of $g(x)$, $D(g)$ is positive and Routh-Hurwitz conditions are satisfied , i.e.
$D(g)>0,\ b_1>0, \ b_3>0$, then the endemic equilibrium $F_{*}$ locally asymptotically stable.

(ii)  If $D(g)<0,\ b_1\geq 0,\ b_2\geq0, \ b_3>0, \ \alpha<2/3$, then the endemic equilibrium $F_{*}$ locally asymptotically stable.

(iii)If $D(g)<0,\ b_1>0,\ b_2>0, \ b_1b_2=b_3, \ \alpha\in (0,1),$ then the endemic equilibrium $F_{*}$ locally asymptotically stable.

(iv) If $D(g)<0,\ b_1<0,\ b_2<0, \ \alpha>2/3$, then the endemic equilibrium $F_{*}$ is unstable.
\end{proposition}

\section{Numerical Simulations}
Analytical solutions of nonlinear coupled fractional order differential equations are as difficult as the classical integer order system of nonlinear equations. Therefore, there is need to consider numerical discretizations and approximations for such systems or models. Over the past years, several numerical schemes have been developed to solve fractional order models. One of the powerful approximation methods that has been considered by many authors \citep[see, e.g,][]{Ahmed2007logistic, OZLAP2011, wow40, sheen2015, Okyere2016a} is the efficient Adams-type predictor-corrector method developed by \cite{wow3, wow4}. \cite{Tao2009} have considered a rigorous error analysis of this numerical scheme.\\

For our numerical approximations, we re-define the variables in the model problem~(\ref{SIRmalaria2}) as:

Let $s_h=X,  \ i_h=Y,  \ r_h=Z,  \ s_v=S,  \ i_v=I.$

Therefore following the approach in \citep{wow3, wow4, sheen2015, wow40} and ignoring the details in the derivation of the scheme, we obtain the discretized system corresponding to the fractional order malaria model~(\ref{SIRmalaria2})as follows:

\begin{equation*}
  \begin{split}
   X_{k+1}= & X_0+\frac{1}{\Gamma (\alpha)} \sum_{j=0}^{k}a_{_{j,k+1} }\biggl[\lambda^{\alpha}_{h}(1-X_j)-a^{\alpha}bm X_{j} I_j+\nu^{\alpha} Y_j+\gamma^{\alpha} Z_j+\delta^{\alpha} X_{j} Y_{j}\biggr]\\ + & \frac{1}{\Gamma (\alpha)}\biggl[\lambda^{\alpha}_{h}(1-X^{p}_{k+1})-a^{\alpha}bm X^{p}_{k+1} I^{p}_{k+1}+\nu^{\alpha} Y^{p}_{k+1}+\gamma^{\alpha} Z^{p}_{k+1}+\delta^{\alpha} X^{p}_{k+1} Y^{p}_{k+1}\biggr]\\
    \\
    Y_{k+1}= & Y_0+\frac{1}{\Gamma (\alpha)} \sum_{j=0}^{k}a_{_{j,k+1} }\biggl[a^{\alpha}bm X_j I_j-(\nu^{\alpha} +r^{\alpha}+\lambda^{\alpha}_{h}+\delta^{\alpha} )Y_j-\delta^{\alpha} Y^2_{j}\biggr]\\ + & \frac{1}{\Gamma (\alpha)}\biggl[a^{\alpha}bm X^{p}_{k+1} I^{p}_{k+1}-(\nu^{\alpha} +r^{\alpha}+\lambda^{\alpha}_{h}+\delta^{\alpha} )Y^{p}_{k+1}+\delta^{\alpha} (Y^{P}_{K+1})^2\biggr]\\ \\
      Z_{k+1}= & Z_0+\frac{1}{\Gamma (\alpha)} \sum_{j=0}^{k}a_{_{j,k+1} }\biggl[r Y_{j}-(\gamma^{\alpha}+\lambda^{\alpha}_{h})Z_j+\delta Y_j Z_j\biggr] + \frac{1}{\Gamma(\alpha)}\biggl[r Y^{p}_{k+1}-(\gamma^{\alpha}+\lambda^{\alpha}_{h})Z^{p}_{k+1}+\delta Y^{p}_{k+1} Z^{p}_{k+1}\biggr]\\\\
     S_{k+1}= & S_0+\frac{1}{\Gamma (\alpha)} \sum_{j=0}^{k}a_{_{j,k+1} }\biggl[\lambda^{\alpha}_{v}(1-S_J)-a^{\alpha}c Y_j S_j\biggr] + \frac{1}{\Gamma (\alpha)}\biggl[\lambda^{\alpha}_{v}(1-S^{p}_{k+1})-a^{\alpha}c Y^{p}_{k+1} S^{p}_{k+1}\biggl]\\\\
    I_{k+1}= & I_0+\frac{1}{\Gamma (\alpha)} \sum_{j=0}^{k}a_{_{j,k+1}}\biggl[a^{\alpha}c S_j Y_j-\lambda^{\alpha}_{v}I_j\biggr] + \frac{1}{\Gamma (\alpha)}\biggl[a^{\alpha}c S^{p}_{k+1} Y^{p}_{k+1}-\lambda^{\alpha}_{v}I^{p}_{k+1}\biggl]
  \end{split}
\end{equation*}

where
\begin{equation*}
  \begin{split}
   X^{p}_{k+1}= & X_0+\frac{1}{\Gamma (\alpha)} \sum_{j=0}^{k}b_{_{j,k+1} }\biggl[\lambda^{\alpha}_{h}(1-X_j)-a^{\alpha}bm X_{j} I_j+\nu^{\alpha} Y_j+\gamma^{\alpha} Z_j+\delta^{\alpha} X_{j} Y_{j}\biggr]\\
    \\
    Y^{p}_{k+1}= & Y_0+\frac{1}{\Gamma (\alpha)} \sum_{j=0}^{k}b_{_{j,k+1} }\biggl[a^{\alpha}bm X_j I_j-(\nu^{\alpha} +r^{\alpha}+\lambda^{\alpha}_{h}+\delta^{\alpha} )Y_j-\delta^{\alpha} Y^2_{j}\biggr]\\
    \\
      Z^{p}_{k+1}= & Z_0+\frac{1}{\Gamma (\alpha)} \sum_{j=0}^{k}b_{_{j,k+1} }\biggl[r Y_{j}-(\gamma^{\alpha}+\lambda^{\alpha}_{h})Z_j+\delta Y_j Z_j\biggr] \\\\
     S^{p}_{k+1}= & S_0+\frac{1}{\Gamma (\alpha)} \sum_{j=0}^{k}b_{_{j,k+1} }\biggl[\lambda^{\alpha}_{v}(1-S_J)-a^{\alpha}c Y_j S_j\biggr] \\\\
    I^{p}_{k+1}= & I_0+\frac{1}{\Gamma (\alpha)} \sum_{j=0}^{k}b_{_{j,k+1} }\biggl[a^{\alpha}c S_j Y_j-\lambda^{\alpha}_{v}I_j\biggr]
  \end{split}
\end{equation*}

$ a_{_{j,k+1} }=\dfrac{h^{\alpha}}{\alpha (\alpha +1)}
  \begin{cases}
    k^{\alpha+1}-(k-\alpha)(k+1)^{\alpha}       & \quad \text{if } j=0\\
    (k-j+2)^{\alpha+1}+  (k-j)^{\alpha+1}-2(k-j+1)^{\alpha+1} &   \quad \text{if } 1\leq j\leq k\\
    1  & \quad \text{if } j=k+1
  \end{cases}
$
and

$b_{_{j,k+1} }=\frac{h^{\alpha}}{\alpha}\left[(k-j+1)^{\alpha}-(k-j)^{\alpha}\right],  \ \ \ 0\leq j\leq k $\\

We then implement our numerical scheme in Matlab and the simulation results are discussed in the next section.

\section{Results and Discussion}
In this section, we discuss the endemic trajectories of our numerical solutions with model parameter values adapted from the work by \cite{Tumwiine2007a} and basic reproduction number $R_o=1.5$. The simulated solutions for the scaled human population $s_h(t), \ i_h(t)$ and $r_h (t)$ are shown in Figure \ref{fg1} and that of the scaled mosquito populations $s_v (t), \ i_v (t)$  are shown in Figure \ref{fg2}. In Figures~\ref{fg3} and \ref{fg4}, we considered four different values of the fractional $\alpha$, $\alpha=1,\ 0.99, \ 0.95\ 0.90$. The oscillatory behavior of the fractional order malaria model~(\ref{SIRmalaria2}) are interesting and can be compared to the research work be \cite{Tumwiine2007a}. It is important to observe that when $\alpha=1$, the model problem~(\ref{SIRmalaria2}) is equivalent to the classical initial value problem~(\ref{SIRmalaria2}) and  that as the fractional order $\alpha$ increases the behaviour of the fractional order model solutions approaches that of the integer order model. Phase plane portrait is displayed in Figure~\ref{fg5} to describe the dynamics between susceptible fractions of human hosts and infected fractions of mosquito vectors.
\newpage
\begin{figure}[h!]
\centering
\subfigure[]{
\includegraphics[scale=0.6]{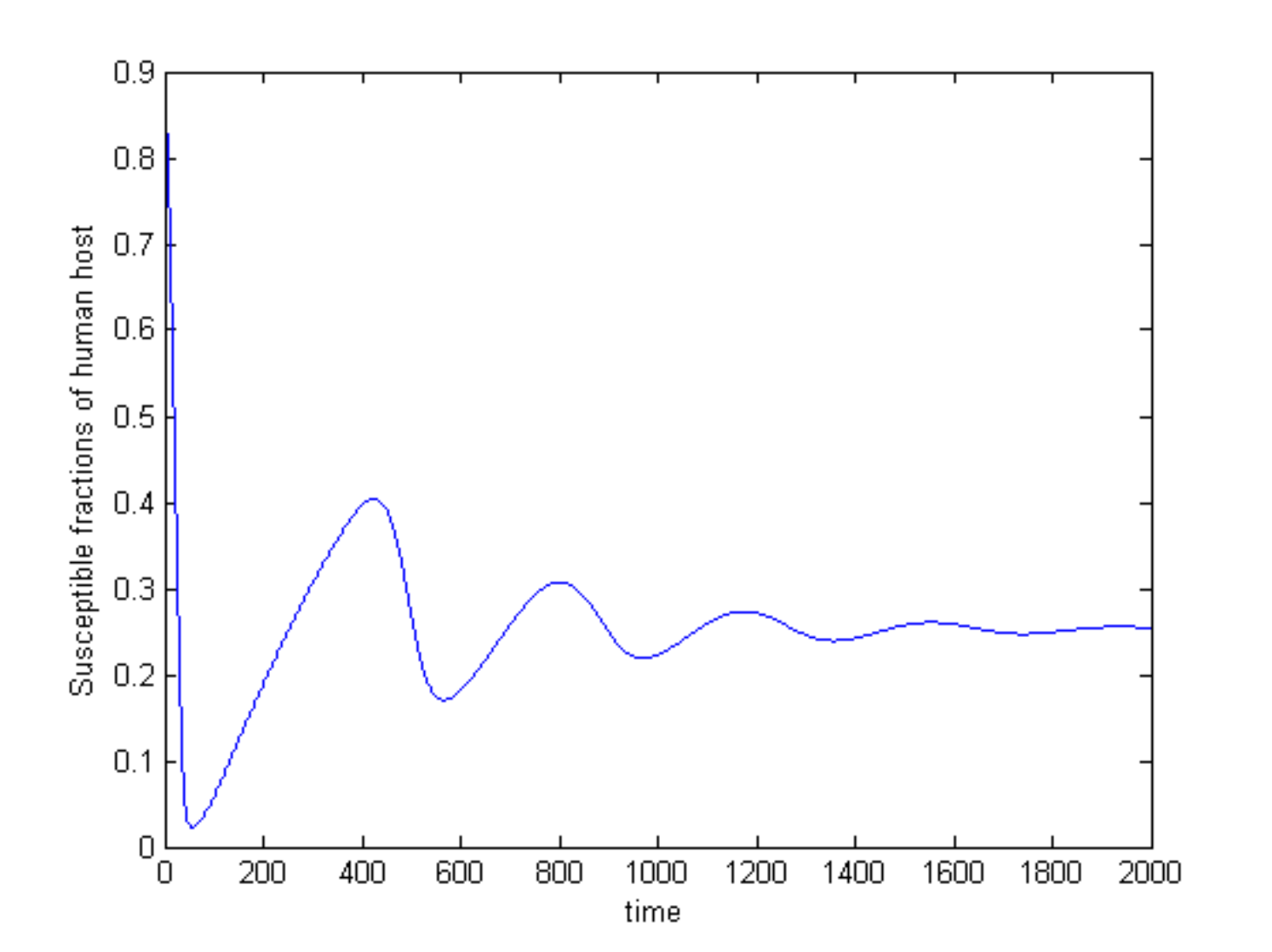}}\hfil
\subfigure[]{
\includegraphics[scale=0.6]{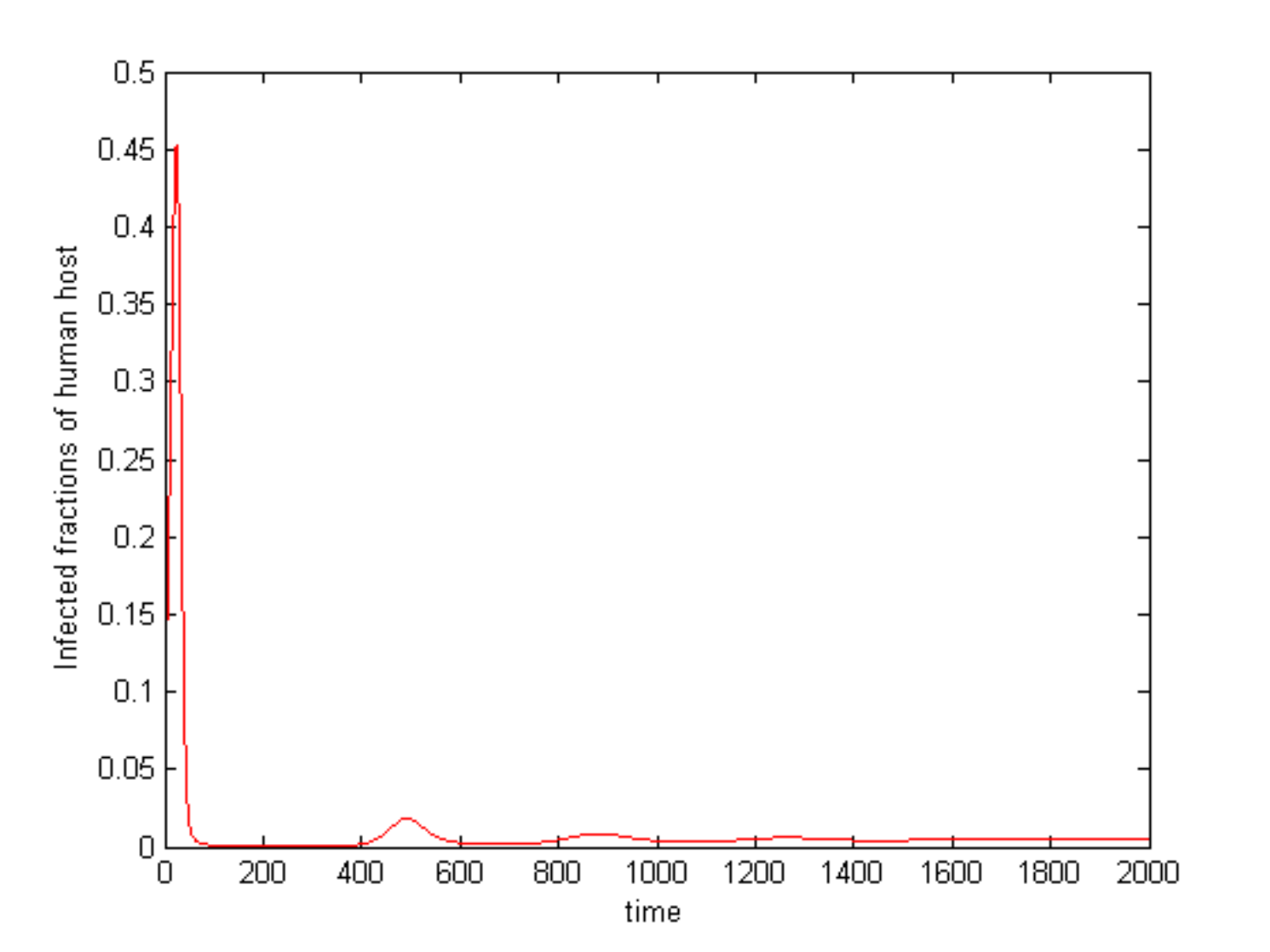}}\hfil
\subfigure[]{\includegraphics[scale=0.6]{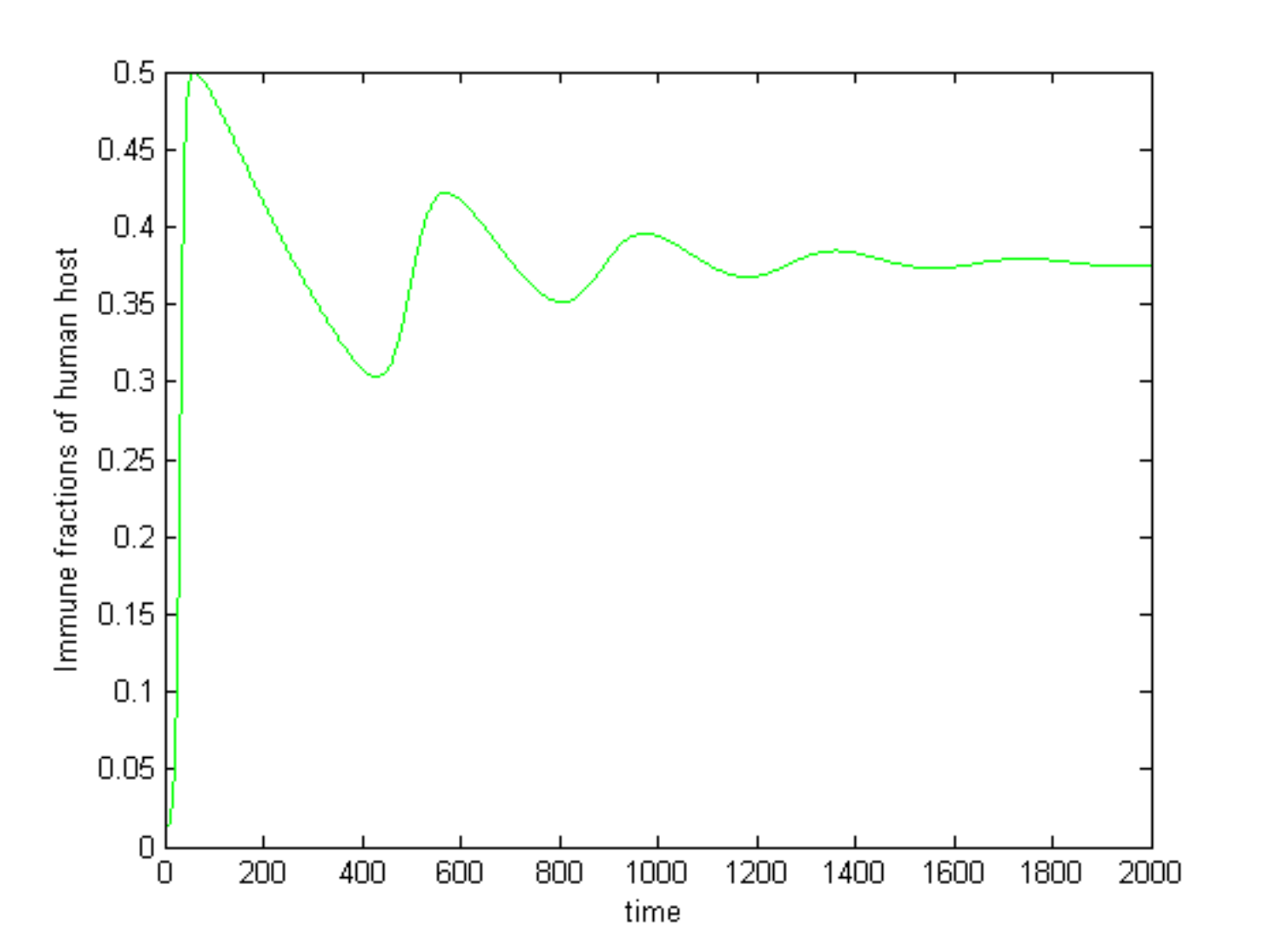}}
\caption{Solutions of fractional order model with $\alpha=1$.}
\label{fg1}
\end{figure}

\begin{figure}[h!]
\centering
\subfigure[]{
\includegraphics[scale=0.6]{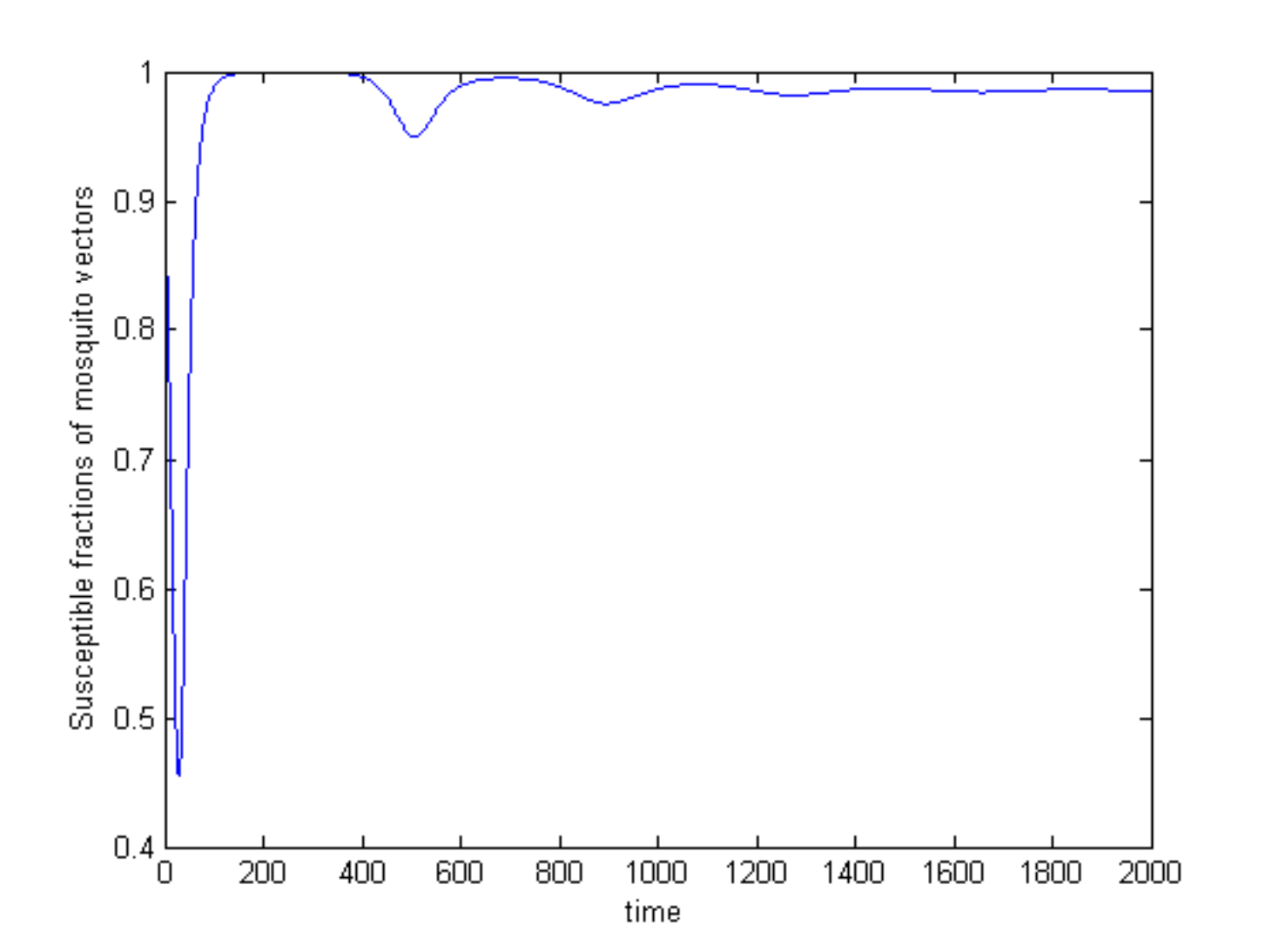}}\hfil
\subfigure[]{\includegraphics[scale=0.6]{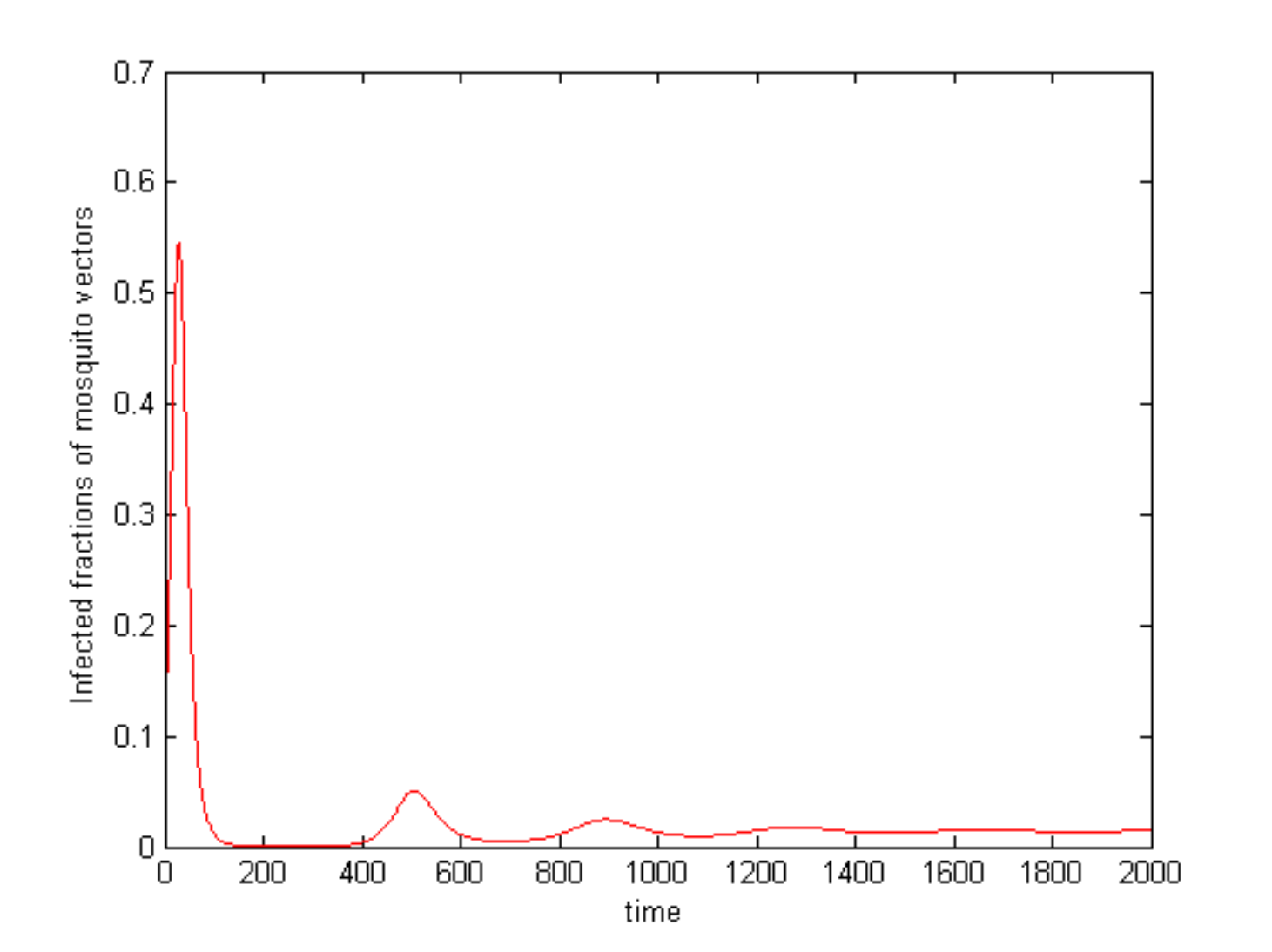}}
\caption{Solutions of fractional order model with $\alpha=1$.}
\label{fg2}
\end{figure}

\begin{figure}[h!]
\centering
\subfigure[]{
\includegraphics[scale=0.6]{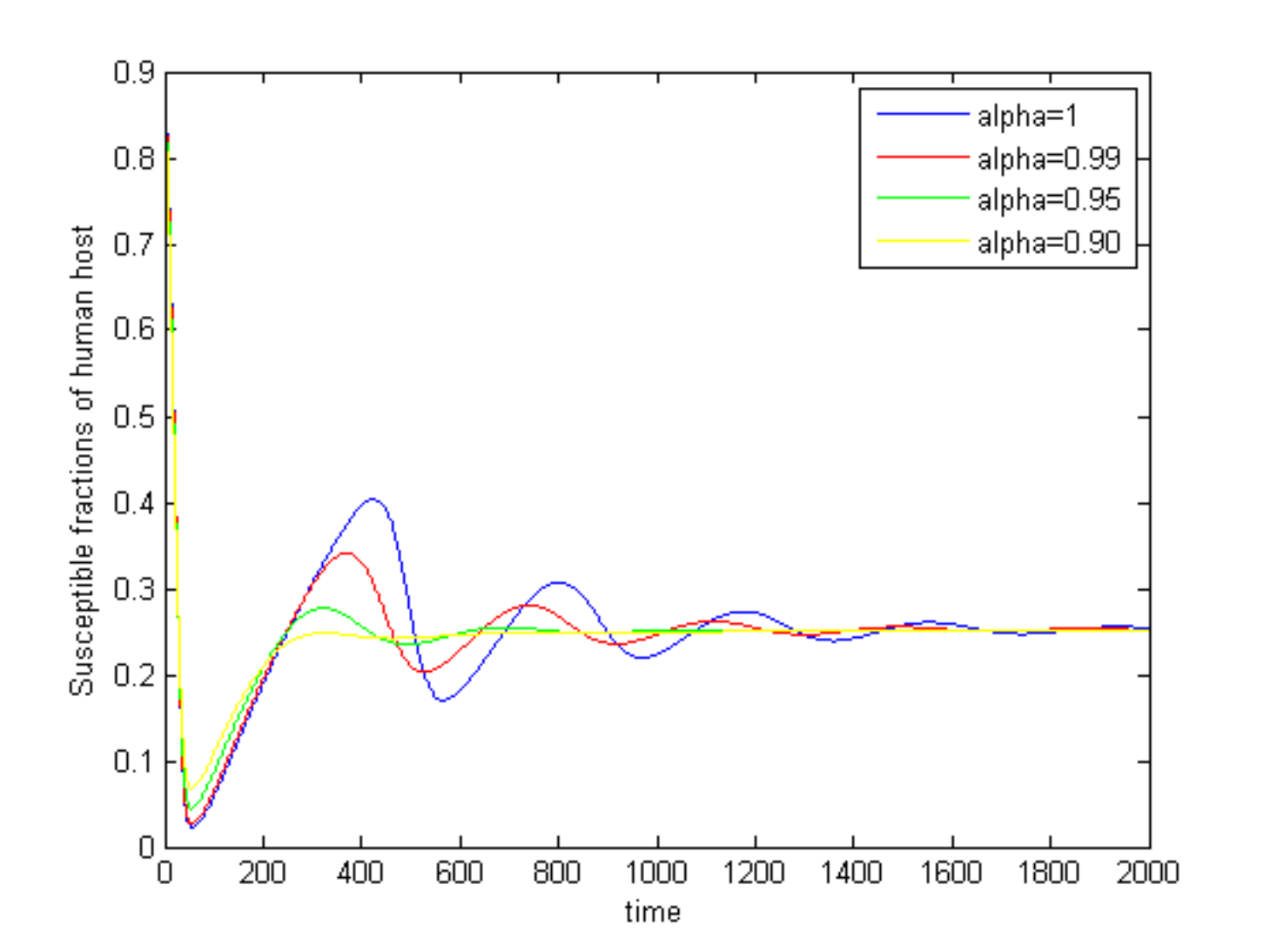}}\hfil
\subfigure[]{
\includegraphics[scale=0.6]{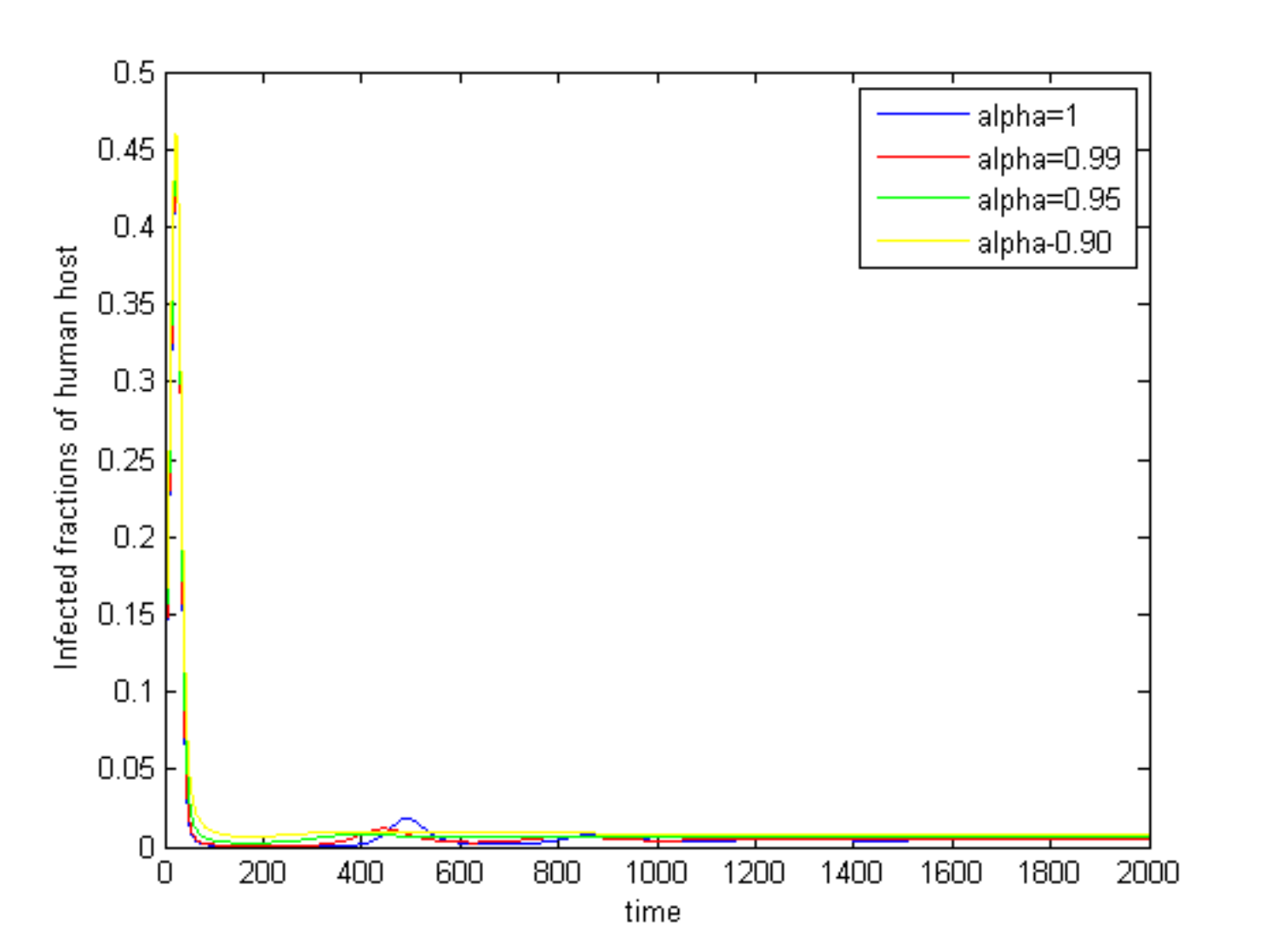}}\hfil
\subfigure[]{\includegraphics[scale=0.6]{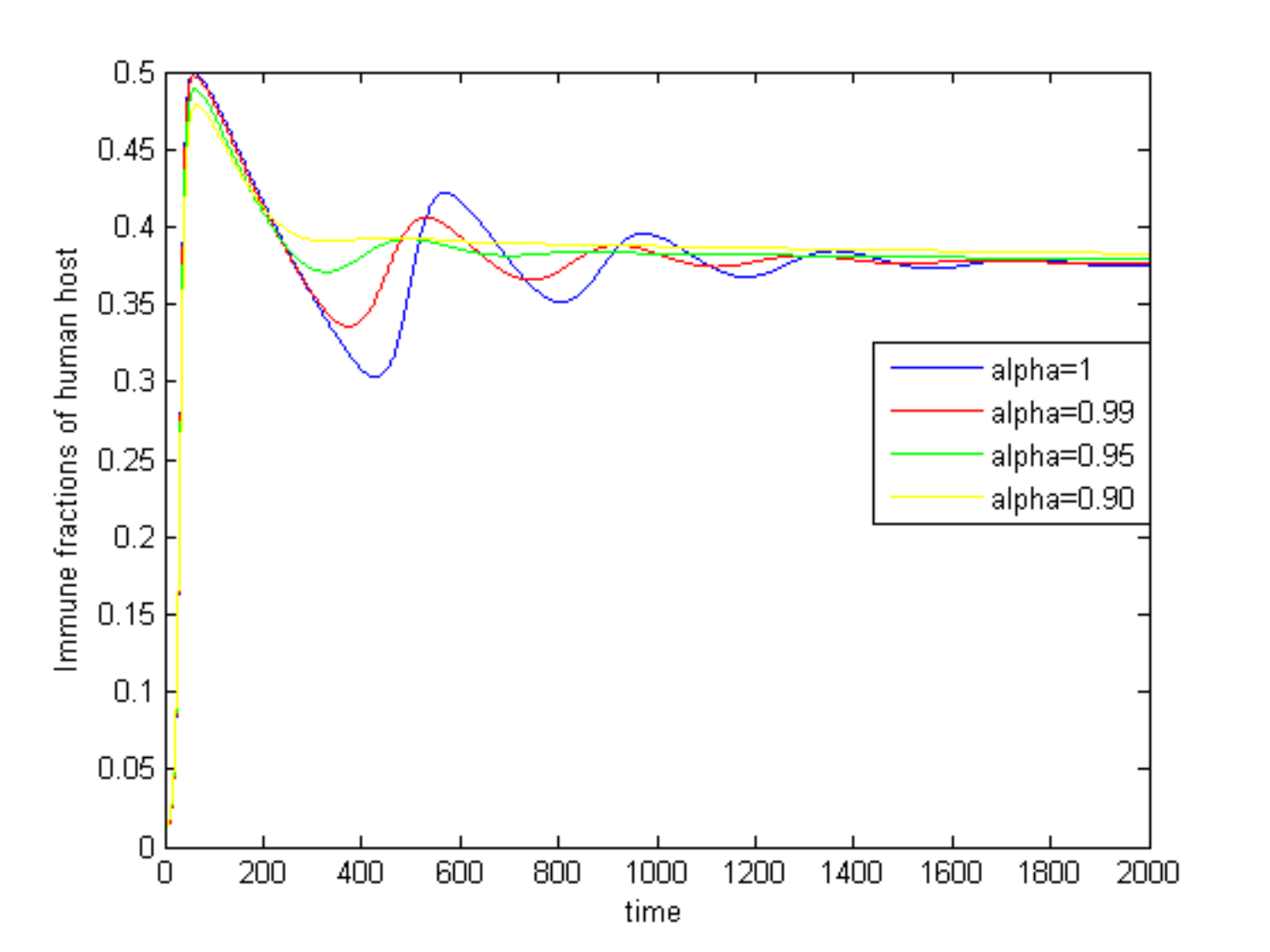}}
\caption{Solutions of fractional order model with $\alpha=1, 0.99, 0.95, 0.90$.}
\label{fg3}
\end{figure}

\begin{figure}[h!]
\centering
\subfigure[]{
\includegraphics[scale=0.6]{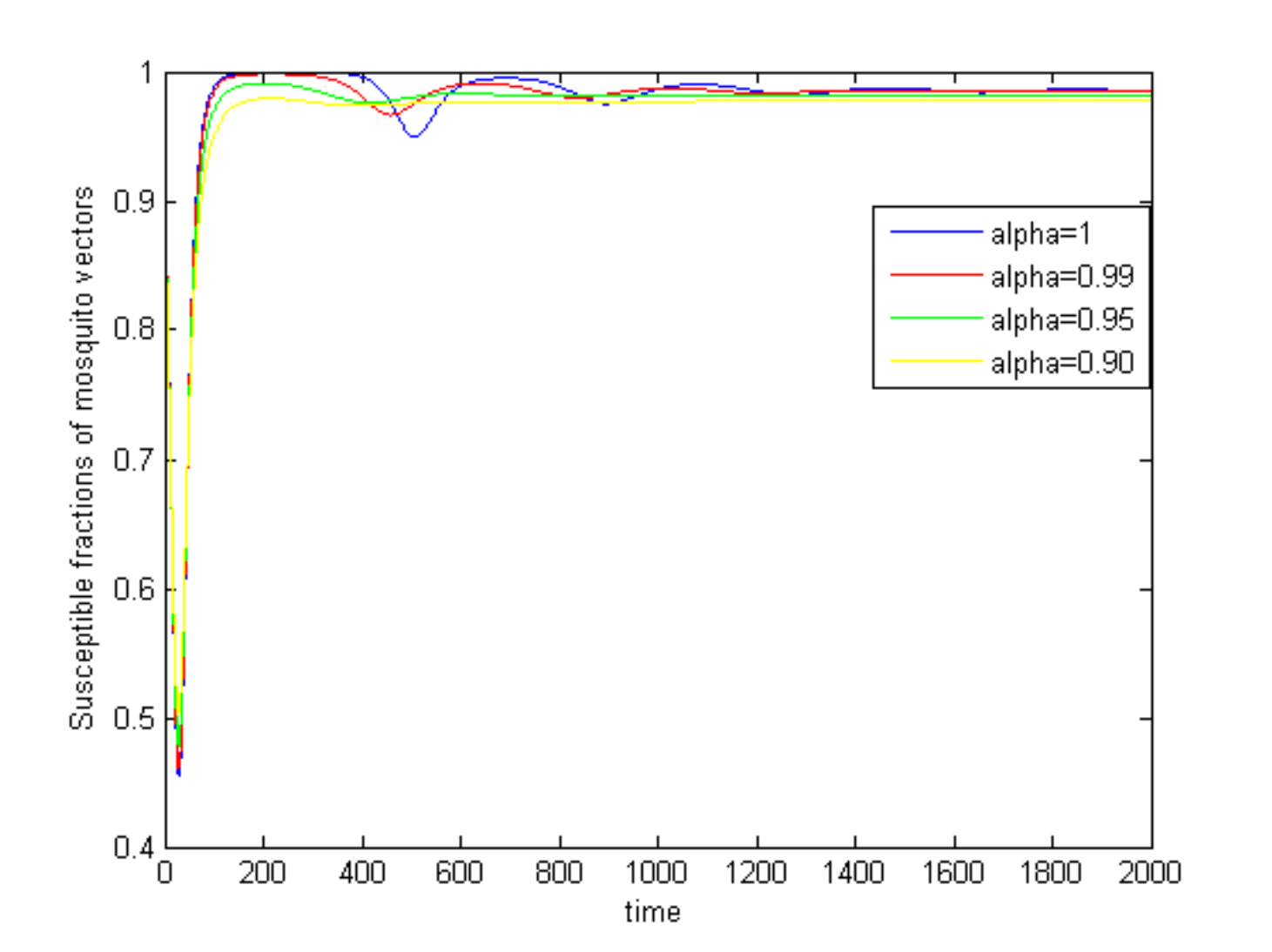}}\hfil
\subfigure[]{\includegraphics[scale=0.6]{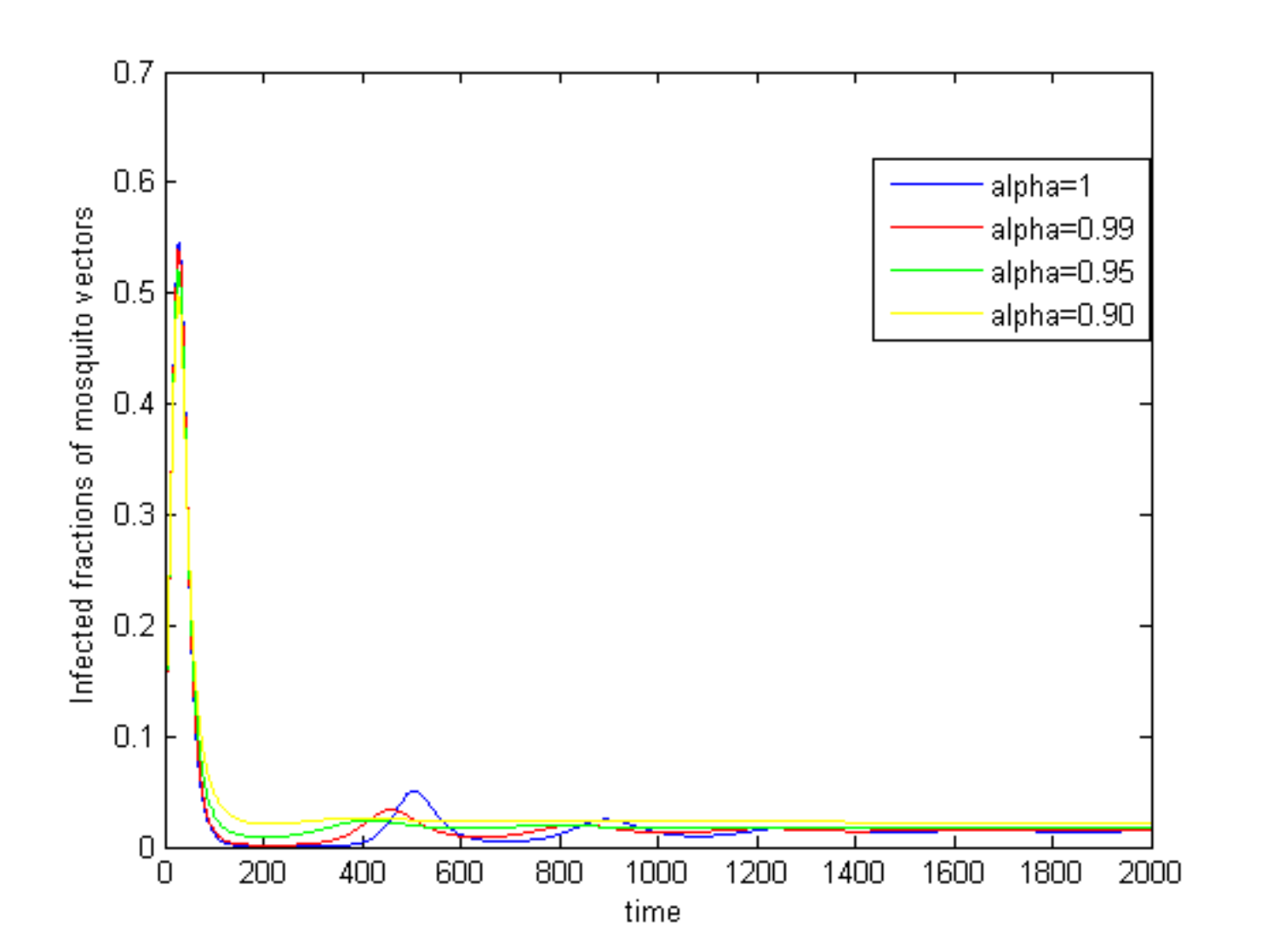}}
\caption{Solutions of fractional order model with $\alpha=1$.}
\label{fg4}
\end{figure}

\begin{figure}[h!]
\centering
\includegraphics[scale=0.8]{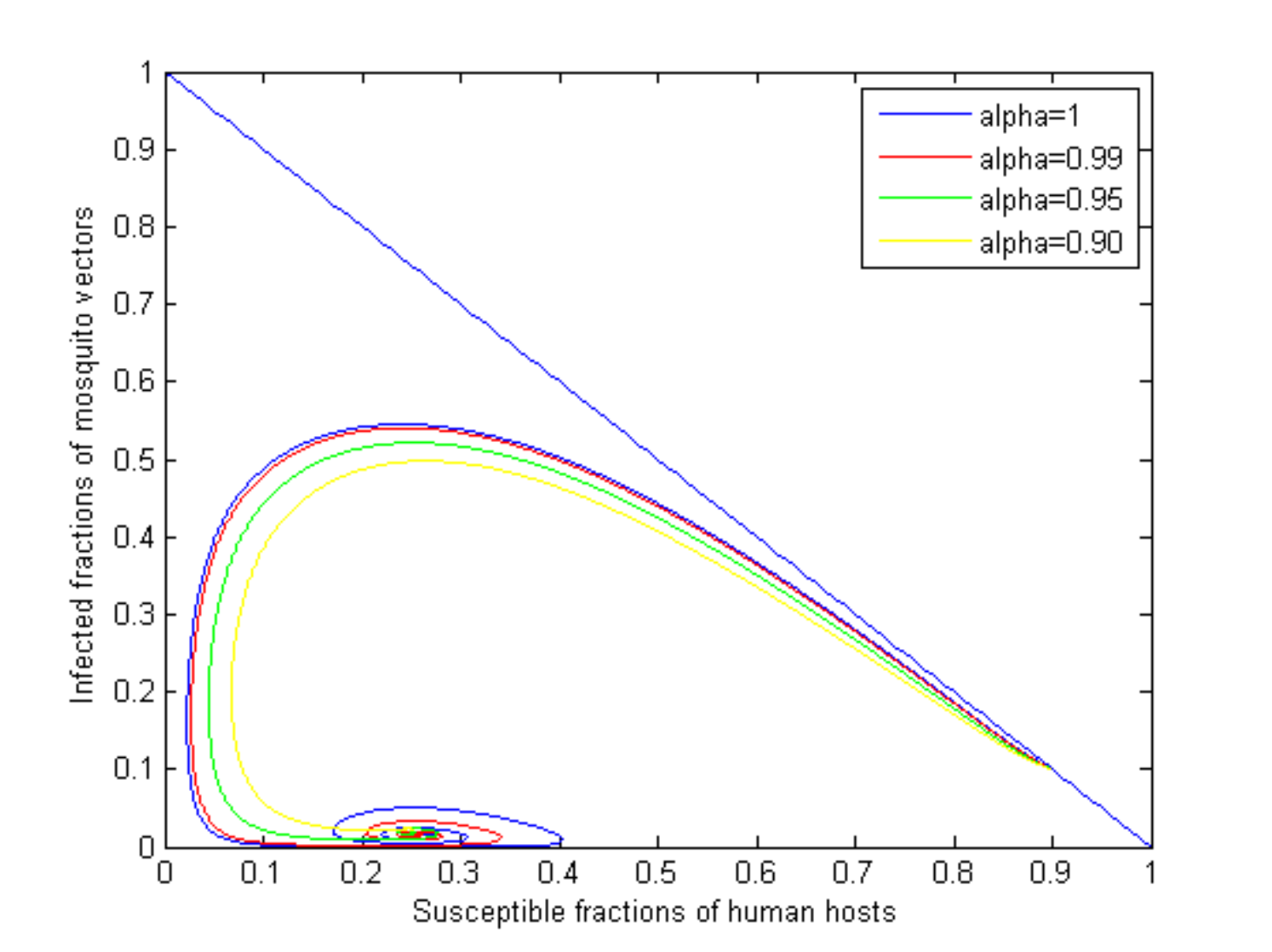}
\caption{Phase portrait for fractional order model with $\alpha=1, 0.99, 0.95, 0.90$.}
\label{fg5}
\end{figure}

\section{Conclusion}
In this research, we have proposed and studied fractional order model for the integer order model developed and studied by \cite{Tumwiine2007a, Tumwiine2007b}. We determined the basic reproduction number $R_o$ using the method and the notations by \cite{van2002}. We have investigated local asymptotic stability analysis of the model equilibria. Furthermore, we have numerically describe the trajectories of the model using Adams-type predictor-corrector method. We have extended the existing malaria model propose by \cite{Tumwiine2007a, Tumwiine2007b} to include fractional derivatives which is very interesting and important in the field of mathematical biology.


\bibliographystyle{spbasic}      

\bibliography{sample1}   

\end{document}